\documentclass{ISMA_USD2020}

\usepackage{bm}
\usepackage{dsfont}
\usepackage{amssymb}

\usepackage{mathtools}

\usepackage[utf8]{inputenc}
\usepackage[english]{babel}

\usepackage{algpseudocode}
\usepackage{algorithm}
\usepackage{cleveref}
\usepackage{amsthm}
\usepackage{graphicx}
\usepackage{amsmath}
\usepackage[version=4]{mhchem}
\usepackage{siunitx}
\usepackage{longtable,tabularx}
\usepackage{booktabs}

\newtheorem{theorem}{Theorem}
\newtheorem{remark}{Remark}

\newtheorem{cor}{Corollary}
\newtheorem{lem}{Lemma}
\newtheorem{defn}{Definition}

\usepackage{cancel}

\renewcommand{\epsilon}{{\varepsilon}}

\newcommand{\braces}[1]{\left \{ {#1} \right \}}

\newcommand{\ip}[2]{\left <{#1},{#2} \right >}

\usepackage{subcaption}
\usepackage{dsfont}

\usepackage[customcolors]{hf-tikz}
\tikzset{ offset def/.style={
    above left offset={-0.1,0.8},
    below right offset={4.2,-0.65},
  }, eqnbox/.style={
    offset def,
    set fill color=white!20!white!60,
    set border color=blue!60!cyan,
  }, offset def2/.style={
    above left offset={-0.1,0.8},
    below right offset={5,-0.65},
  },  eqnbox2/.style={
    offset def2,
    set fill color=white!20!white!60,
    set border color=blue!60!cyan,
  }
}

\hypersetup{pdftitle  = {Manuscript preparation instructions},
	pdfauthor = {P. Cheema, M. M. Alamdari, G. A. Vio}}

\title{On physical realizability for inverse structural designs:\\
bounding the least eigenvalue of an unknown mass \\matrix}

\author[1,2]{P. Cheema}
\author[3]{M. M. Alamadari}
\author[1]{G. A. Vio}
  
\affil[1]	{The University of Sydney, School of Aerospace, Mechanical, and Mechatronic Engineering, \NewLineAffil NSW 2006, Australia \NewAffil}

\affil[2]	{Data61 CSIRO, \NewLineAffil NSW 2015, Australia \NewAffil}

\affil[3]	{University of New South Wales, School of Civil and Environmental Engineering, \NewLineAffil Sydney,
NSW 2052, Australia}
		
\date{}

\begin{document}



\abstract{In the field of structural engineering analysis, a common requirement is to calculate the modal frequencies of a structure that has undergone an update, either naturally (such as from material degradation), or due to man-made influences (by placing point masses along a structure). In addition to this requirement, it is common to only have access to truncated modal testing results. In this paper, we derive analytical bounds for the first eigenvalue of a completely unobserved mass matrix for linear elastic systems. Doing so allows engineers to proceed with modifying linear elastic systems, without requiring direct access to the mass matrix. This is because it is often difficult to know exactly what negative mass perturbations are allowable, given that the full mass matrix is an unknown quantity. Ultimately, the analysis in this paper will proceed by assuming only access to the left and right eigenvectors of the underlying system - which are both possible to obtain via physical experiments, so that the bounds are not only physically realizable, but also practically realizable.}

\maketitle


\section{Introduction}

The field of structural design modification requires altering the properties of a system, and then observing the subsequent effects upon the eigenspectrum. The forward problem of structural modification is well-understood, as it involves observing the eigenvalues which result from a prescribed modification \cite{sestieri2000structural,cheema2020new}. The inverse problem however, is less so. In particular, inverse structural modification involves  determining the set of permissible modifications which generate a prescribed eigenspectrum \cite{sestieri2000structural,he2001structural}. Generally, inverse design problems can be approached from the perspective of analytical modeling \cite{he2001structural,navd2007structural,gladwell1999inverse,mottershead1993model}, passive modeling  \cite{ouyang2015passive}, and or active modeling \cite{inman2001active}. Regardless of the chosen method, most work to constrain the mass modifications as \textit{being positive only}. This is because positive constraints help ensure that the derived answer is physically realizable. This is because typically one does not have access to the full mass matrix, meaning that it is difficult to know the precise extent to which negative mass modifications can be physically made. However, this means that the total space of achievable configurations is much smaller than what it possibly could be. In regards to trying to solve this problem, some attempts have been made to understand the lower bounds on finite element system mass matrices \cite{wathen1987realistic}, as well as stiffness matrices \cite{hu1998bounds}. Moreover, a lot of work on bounding the eigenvalues of a modified system based on perturbations has been done Ram \& Braun wherein they find upper and lower bounds for the eigenvalues of a modified structure based on a truncated modal system \cite{ram1990structural}. However, all such methods require at least full access to a full mass matrix, which limits the applicability of such methods to more realistic settings. Ultimately, it will be the purpose of this paper to demonstrate a method for obtaining a lower bound on the least eigenvalue of the modified mass matrix for a linear-elastic system, even if both the original mass and stiffness matrices are \textit{completely unobserved}. In order to achieve these bounds, only access to the left, and right eigenvectors  of the original system will be required - which can all be obtained from practical experimentation \cite{bucher1997left}.


\section{Problem Set-Up}
Consider an $n\in\mathbb{Z}$ degree of freedom (dof) linear-elastic system as follows: 
\begin{align}
\bm{M}\bm{\ddot{x}} + \bm{K}\bm{x} = \bm{0}
\end{align}
where $\bm{M}\in \mathbb{R}^{n\times n}$ is known as the \textit{mass matrix}, $\bm{K}\in \mathbb{R}^{n\times n}$ is known as the \textit{stiffness matrix}, and $\bm{x}\in \mathbb{R}^{n}$ is a state vector. If the solutions of $\bm{M}\bm{\ddot{x}} + \bm{K}\bm{x} = \bm{0}$ are sinusoidal in nature, then it is equivalent to solving the following generalized eigenvalue problem: 
\begin{align}
(\bm{K} - \lambda_i \bm{M})\bm{v_i} = \bm{0}
\end{align}
where $i = 1,2,...,n$. For this system, $\bm{v}_i\in\mathbb{R}^n$ denotes the $i$-th eigenvector, and $\lambda_i\in\mathbb{R}$ is the $i$-th eigenvalue. It is possible to assemble these $n$ eigenvalue equations into a single matrix form, by constructing $\bm{\Lambda}\in\mathbb{R}^{n\times n}$, and $\bm{V}\in\mathbb{R}^{n\times n}$, where, $\bm{\Lambda} = \textrm{diag} (\lambda_i)_{n \times n}$, and $\bm{V} = [\bm{v_1}, \bm{v_2}, ..., \bm{v_n}]$. This results in the following system:
\begin{align}\label{eqn:VK_MVLambda}
    \bm{K}\bm{V} = \bm{M}\bm{V}\bm{\Lambda}
\end{align}
in which $\bm{\Lambda}$ is known as the \textit{spectral matrix}, and $\bm{V}$ is the \textit{modal matrix}. In this paper it is assumed that the eigenvectors are mass normalized, meaning, $\bm{V}^{\intercal}\bm{M}\bm{V} = \bm{I}$, and $\bm{V}^{\intercal}\bm{K}\bm{V} = \bm{\Lambda}$, where $\bm{I}$ denotes the identity matrix. Moreover, the notation: $(\bm{M},\bm{K})$ will be used to denote the eigenvalue problem for the mass-stiffness system implied by matrices $\bm{M}$, and $\bm{K}$ (that is, Equation \ref{eqn:VK_MVLambda}). From this notation, the corresponding system for: $(\bm{I},\bm{M})$ is defined as the following: 
\begin{align}\label{eqn:mass_system}
    (\bm{M}-w_i \bm{I})\bm{\phi}_i = \bm{0}
\end{align}
which has $i$-th eigenvalue $w_i$, and eigenvector $\bm{\phi}_i$ respectively. This is an important construction for this paper, as we are concerned with \textit{bounding the least eigenvalue} of an unobserved mass matrix in order to impose constraints when performing inverse modeling. In terms of inverse design, consider now the application of some alteration to the $(\bm{M},\bm{K})$ system. This may manifest through a design change, a change in material properties, or even a change in the new operating conditions, and they will be denoted by: $\Delta\bm{M}$ and $\Delta\bm{K}$. Applying these to the $(\bm{M},\bm{K})$ system, we arrive at the $(\bm{M}+\Delta\bm{M},\bm{K}+\Delta\bm{K})$ eigenvalue problem,
\begin{equation} \label{eqn:full_sys}
(\bm{K} + \Delta \bm{K})\bm{V}^{\star} = (\bm{M} + \Delta \bm{M})\bm{V}^{\star}\bm{\Lambda}^{\star}
\end{equation}
where $\bm{\Lambda}^{\star},\bm{V}^{\star}\in\mathbb{R}^{n\times n}$  denote the new spectral, and modal matrices for the modified system respectively. Ultimately, the aim here will be the following:

\textbf{Aim:} \textit{Under what conditions is it possible to have negative values enter into $\Delta \bm{M}$, such that the underlying $(\bm{M}+\Delta\bm{M},\bm{K}+\Delta\bm{K})$ system remains physically realizable?}

In order to answer this question, it will be necessary to define: (i) A notion of physical realizability, and (ii) Conditions under which physical realizability is plausible. In order to make the problem amenable to practical settings, we will assume no access to the $\bm{M}$, and $\bm{K}$ matrices.

\section{Establishing Physical Realizability}

In the absence of the $\bm{M}$, and $\bm{K}$ matrices, it is difficult to know what are the allowable negative values allowed for $\Delta \bm{M}$. In order to help approach this problem, it is necessary to (i) Establish a notion of physical realizability, and (ii) Develop bounds based on the prescribed definition of physical realizability. Therefore, this section works to provide an intuitive mathematical statement to capture the notion of physical realizability, based on system kinetic energy (KE). This idea is made clear in Definition \ref{def:KE}, and the mathematical implication of this definition is made clear in Remark \ref{lem:G}.

\begin{defn}\label{def:KE} A linear mass-spring system that is claimed as being physically realizable requires at least, to possess positive KE.
\end{defn}

\begin{remark} \label{lem:G} An  $(\bm{M},\bm{K})$ system has positive KE if the eigenvalues of the $(\bm{I},\bm{M})$ system are positive. That is, if $w_i>0$ for all $i = 1,...,n$, then $\text{KE}\left[(\bm{M},\bm{K})\right]>0$.
\end{remark}

\begin{proof}
As shown in Equation \ref{eqn:mass_system}, the eigenpairs for the mass-only system: $(\bm{I},\bm{M})$, are $(w_i, \bm{\phi}_i)$, with eigenvalues: $w_i$, and eigenvector: $\bm{\phi}_i$. Consider the following diagonalization:
\begin{align*} 
\bm{M} = \bm{\Phi}^{-1} \bm{D \Phi}, 
\end{align*} 
where $\bm{\Phi}=[\bm{\phi}_1, \ldots, \bm{\phi}_n]$ is the matrix of eigenvector columns and $\bm{D} = \text{diag}(w_1,\ldots,w_n)$.  Since $\braces{\bm{\phi}_i}_{i=1}^n$ forms a basis over $\mathbb{R}^n$, any $\dot{\bm{x}} \in \mathbb{R}^n$ velocity vector can be expressed as $\dot{\bm{x}} = \sum_{i=1}^n a_i \bm{\phi}_i$, where $a_i\in\mathbb{R}$ are constants. Now, $\forall \dot{\bm{x}} \in \mathbb{R}^N$, $\text{KE}\left[(\bm{M},\bm{K})\right]$, w.r.t. $\dot{\bm{x}}$ can be defined as follows:  
\begin{align} \text{KE}\left[(\bm{M},\bm{K});\dot{\bm{x}}\right] &\coloneqq \frac{1}{2} \ip{\dot{\bm{x}}}{\bm{M}\dot{\bm{x}}} \nonumber\\ &= \ip{\sum_{i=1}^N a_i \bm{\phi}_i}{\sum_{j=1}^N a_j \bm{M \phi}_j}\nonumber\\ &\stackrel{(i)}{=} \sum_{i=1}^N \sum_{j=1}^N a_i a_j \ip{\bm{\phi}_i}{ w\bm{\phi}_j} \\ &\stackrel{(ii)}{=}\sum_{i=1}^N w_i a_i^2 \nonumber\\  & > 0, \text{ if $w_i > 0$.} \nonumber 
\end{align} 
where (i) follows from linearity of the inner product, and (ii) follows from the orthogonality property of eigenvectors.
\end{proof}

\textbf{Summary: } \textit{What Remark \ref{lem:G} says, is that if the eigenvalues of $\bm{M}$ are positive, then the corresponding $(\bm{M},\bm{K})$ system is physically realizable through a KE argument. In other words, if $\bm{M}$ is positive definite (denoted: $\bm{M} \succ 0$) then the system is physically realizable. We will use this idea to guide our lower bounds.}

\section{The Eigenvalue Lower Bound based on the Left Eigenvector}

In this section, we will establish a very general lower-bound on the first eigenvalue of an unobserved mass matrix, based only on the first left, and right eigenvectors of the $(\bm{M},\bm{K})$ system. The underlying motivation to this section will be based on Weyl's theorem of interlacing eigenvalues as show in Theorem \ref{thm:Weyl}.

\begin{theorem}\textbf{[Weyl's Theorem \cite{zheng2020inertia}]}\label{thm:Weyl}
Let $\bm{A,B} \in \mathbb{R}^{n\times n}$ be Hermitian matrices. Then:
\begin{align} \label{eqn:Weyl}
    \lambda_i(\bm{A}) + \lambda_1(\bm{B}) \leq \lambda_i(\bm{A+B}) \leq \lambda_i(\bm{A}) + \lambda_n(\bm{B})
\end{align}
    \end{theorem}
For the context of this paper, we can let $\bm{A}=\bm{M}$, and $\bm{B}=\Delta \bm{M}$\footnote{Note that Weyl's inequality applies for Hermitian systems, and for most cases $\bm{M}$ and $\bm{K}$ are indeed Hermitian matrices.}. Now from observing Equation \ref{eqn:Weyl}, it can be seen that if $i=1$ is selected, then we have: $\lambda_1(\bm{M}) + \lambda_1(\Delta \bm{M}) \leq \lambda_1(\bm{M}+\Delta \bm{M})$. Since positive KE is desired (Remark \ref{lem:G}), it can be seen that we need: $\lambda_1(\bm{M}+\Delta \bm{M})>0$, in turn requiring that: $\lambda_1(\Delta \bm{M})< -\lambda_1(\bm{M})$. This provides the maximal possible constraint for allowing $\Delta \bm{M}$ to have negative values in the matrix: it needs to be smaller than the magnitude of the least eigenvalue of the mass matrix; which is a rather intuitive conclusion. The main problem with this conclusion, however, is that \textit{we do not have access to $\bm{M}$}! Thus, we need to develop some way to implicitly obtain information from the mass matrix. This can be achieved with knowledge of the left eigenvectors. In fact it can be shown that $\bm{g}_i = \bm{M}\bm{v}_i$, where $\bm{g}_i$, and $\bm{v}_i$, are the $i$-th left, and right eigenvectors respectively. This is clarified in Lemma \ref{lem:LR}.

\begin{lem}\textbf{[Left-Right Eignevector Relationship]} \label{lem:LR}
For the $(\bm{M},\bm{K})$ system, $\bm{g}_i = \bm{M}\bm{v}_i$, where $\bm{g}_i$, and $\bm{v}_i$, are the $i$-th left, and right eigenvectors respectively.
\end{lem}
\begin{proof}
Recall $(\bm{M},\bm{K})\coloneqq \left(\bm{Kv}_i=\lambda_i\bm{Mv}_i\right)$. From this definition, the right-eigenvector system can be expressed as: $\bm{\left(M^{-1}K\right)v}_i=\lambda_i\bm{v}_i$, and the left-eigenvector system can be exressed as: $ \bm{g}_i^{\intercal}\bm{\left(M^{-1}K\right)}=\lambda_i\bm{g}_i^{\intercal}$. Assume there exists an $\bm{A}$, such that $\bm{g}_i = \bm{Av}_i$, so that $\bm{v}_i$, and $\bm{g}_i$ are linearly related. The aim will be to identify what form of $\bm{A}$ is acceptable. Substituting  $\bm{g}_i = \bm{Av}_i$, into the left-eigenvector equation gives:
\begin{alignat}{2}
       &\qquad& (\bm{Av}_i)^{\intercal}(\bm{M}^{-1}\bm{K}) &= \lambda_i (\bm{Av}_i)^{\intercal} \nonumber\\
       \stackrel{(i)}{\iff} && (\bm{M}^{-1}\bm{K})^{\intercal}\bm{Av}_i &= \lambda_i (\bm{Av}_i) \nonumber\\
        \iff && (\bm{K}^{\intercal})(\bm{M}^{\intercal})^{-1}\bm{Av}_i &= \lambda_i (\bm{Av}_i) \nonumber\\
        \iff && \bm{KM}^{-1}\bm{Av}_i &= \lambda_i (\bm{Av}_i) \label{eqn:last_step}
\end{alignat}
where (i) follows from transposing both sides. By observing Equation \ref{eqn:last_step}, the only choice of $\bm{A}$ which will satisfy the $(\bm{M},\bm{K})$ system is $\bm{A} =\bm{M}$.
\end{proof}

Since it is known that the left eigenvector can be obtained experimentally \cite{bucher1997left}, we therefore have a means to implicitly obtain information about the unobserved mass matrix, $\bm{M}$ in practical settings. This relationship can be exploited to then form a lower bound on the smallest eigenvalue of the mass matrix (without ever needing to observe it). This is shown in Theorem \ref{thm:mine}. However, in order to establish Theorem \ref{thm:mine}, it is necessary to establish Remark \ref{rem:3a}.

\begin{remark}  \label{rem:3a}
For eigenvalue system $(\bm{I},\bm{M})$, the modified eigenvalue problem of: $(\bm{I},(\bm{M} - \alpha \bm{I})^{-1})$ has corresponding eigenvalues $\frac{1}{w_i - \alpha}$,  where $\alpha \in \mathbb{R}$ is a scalar.
\end{remark}
\begin{proof}
Firstly if: $\bm{M}\bm{\phi}_i=w_i\bm{\phi}_i$, then $\bm{M}^{-1}\bm{\phi}_i=\frac{1}{w_i}\bm{\phi}_i$. Moreover, if $\bm{M}\bm{\phi}_i=w_i\bm{\phi}_i$, then $(\bm{M}-\alpha\bm{I})\bm{\phi}_i=(w_i-\alpha)\bm{\phi}_i$. From these two observations, the statement of Remark \ref{rem:3a} follows (the eigenvalues are inverted, and translated by $\alpha$).
\end{proof}

We now proceed to the main theorem of this paper, as outlined in Theorem \ref{thm:mine}.

\begin{theorem}\textbf{[Generalized Eigenvalue Lower Bound]} \label{thm:mine}
Let $\text{eig}\left[(\bm{I},\bm{M})\right] \coloneqq \{w_i\}_1^n$. Then $\forall \bm{x} \in \mathbb{R}^n,$ such that $\|\bm{x}\| \neq 0,$, and  $\forall \alpha \in \mathbb{R}$ such that $|w_1 - \alpha| < |w_2 - \alpha|$ we have,

\begin{equation}\label{eqn:Thm_bound}
\tikzmarkin[eqnbox]{a}
    w_1 \geq \alpha - \frac{\|\bm{x}\| \| \bm{\tilde{g}} - \alpha \bm{\tilde{v}}\| }{\ip{\bm{x}}{\bm{\tilde{v}}} },
    \tikzmarkend{a}
\end{equation}

where $\bm{\tilde{g}}$, and $\bm{\tilde{v}}$ are the left and right eigenvectors corresponding to $\min (\text{eig}\left[ (\bm{M},\bm{K})\right])$.  

\end{theorem}

\begin{proof}
Consider,
\begin{align}
     \ip{\bm{x}}{\bm{\tilde{v}}} &\stackrel{(i)}{\leq} \|\bm{x}\| \|\bm{\tilde{v}}\| \nonumber \\
                                 &= \|\bm{x}\| \|(\bm{M}-\alpha \bm{I})^{-1}(\bm{M}-\alpha \bm{I})\bm{\tilde{v}}\| \nonumber\\
                                 &\stackrel{(ii)}{\leq} \|\bm{x}\| \|(\bm{M}-\alpha \bm{I})^{-1} \| \|(\bm{M}-\alpha \bm{I})\bm{\tilde{v}}\|. \nonumber
\end{align}

where (i) follows from the Cauchy-Schwarz inequality, and (ii) from the notion that induced operator norms are sub-multiplicative. In particular, the vector norm $\| \cdot \|: \mathbb{R}^n \rightarrow \mathbb{R}$, is known to induce the following operator norm,
\begin{equation}
    \|\bm{A} \| \coloneqq \sup \{ \|\bm{A} \bm{x} \| : \forall \bm{x}\in\mathbb{R}^n, \|\bm{x}\| = 1 \},
\end{equation}
for all such operators $\bm{A}\in\mathbb{R}^{n \times n}$ \cite{chu2005inverse}. In addition, the operator norm relates to the spectral radius of the underlying operator in the following way: $\|\bm{A}\|= \rho(\bm{A})$, where  $\rho(\bm{A}) \coloneqq \max(|\text{eig}\left[(\bm{I},\bm{A})\right]|)$ is the spectral radius \cite{chu2005inverse}. Thus,
\begin{align*}
      \|(\bm{M}-\alpha \bm{I})^{-1} \| &\coloneqq  \rho((\bm{M}-\alpha \bm{I})^{-1}) \\ 
                                       &\stackrel{(i)}{=} \max(|\text{eig}\left[(\bm{M}-\alpha \bm{I})^{-1}) \right] |) \\
                                       &\stackrel{(ii)}{=} \frac{1}{\min (|w -\alpha|) } \\
                                       &\stackrel{(iii)}{=} \frac{1}{|w_1 -\alpha|}, 
\end{align*}
where (i) Follows from definition of the spectral radius, (ii) Is established from Remark \ref{rem:3a}, and (iii) Holds as long as $|w_1 - \alpha| < |w_2 - \alpha|$. Ultimately, we have found a way to re-write the $\|(\bm{M}-\alpha \bm{I})^{-1} \|$ expression. Therefore we can write, 
\begin{alignat}{2}
       &\qquad& \ip{\bm{x}}{\bm{\tilde{v}}} &\leq \frac{\|\bm{x} \| \| (\bm{M}\bm{\tilde{v}}-\alpha \bm{\tilde{v}}) \|}{|w_1 -\alpha|} \nonumber \\
       \iff && |w_1 -\alpha| &\leq \frac{\|\bm{x} \| \| (\bm{M}\bm{\tilde{v}}-\alpha \bm{\tilde{v}}) \|}{\ip{\bm{x}}{\bm{\tilde{v}}}},
\end{alignat}
where it assumed that $\ip{\bm{x}}{\bm{\tilde{v}}} > 0$. Finally, we notice that $\bm{\tilde{g}} = \bm{M}\bm{\tilde{v}}$ (established from Lemma \ref{lem:LR}), and expanding the absolute value, we obtain   
\begin{equation*}
    w_1 \geq \alpha - \frac{\|\bm{x}\| \| \bm{\tilde{g}} - \alpha \bm{\tilde{v}}\| }{\ip{\bm{x}}{\bm{\tilde{v}}} }.
\end{equation*}
\end{proof}
\textbf{Summary: }\textit{We have established a bound on the least eigenvalue ($w_1$) of a completely unobserved mass matrix, $\bm{M}$. This bound varies as a function of the scalar, $\alpha$, and requires access to only the left and right system eigenvectors, $\bm{\tilde{g}}$, and $\bm{\tilde{v}}$. However, there is ``arbitrary'' $\bm{x}$ left over. The next section will show how to remove $\bm{x}$.}

\section{Practical Considerations}

Since $\bm{x}$ is arbitrary, it is possible to select it as being parallel to $\bm{\tilde{v}}$, so that its presence can be eliminated from Theorem \ref{thm:mine} altogether. This is made clear in Corollary \ref{cor:mine}, which removes $\bm{x}$, whilst studying the case of $w_1 \geq 0$ (the central concern of this paper).

\begin{cor}\textbf{[Appropriate $\bm{\alpha}$ Range]}\label{cor:mine}

\begin{align}
\tikzmarkin[eqnbox2]{b}
    \alpha \geq \frac{\| \bm{\tilde{g}} - \alpha \bm{\tilde{v}}\| }{\|\bm{\tilde{v}}\|}, \quad \text{ for  } w_1\geq 0. 
\end{align}
\tikzmarkend{b}
\end{cor}
\begin{proof}

Consider, 
\begin{align*}
\alpha - \frac{\|\bm{x}\| \| \bm{\tilde{g}} - \alpha \bm{\tilde{v}}\| }{ \ip{\bm{x}}{\bm{\tilde{v}}} } &=  \alpha - \frac{\|\bm{x}\| \| \bm{\tilde{g}} - \alpha \bm{\tilde{v}}\| }{ \|\bm{x}\| \|\bm{\tilde{v}}\| \cos{\theta_{\bm{x},\bm{\tilde{v}} } } } \\
&= \alpha - \frac{ \| \bm{\tilde{g}} - \alpha \bm{\tilde{v}}\| } { \|\bm{\tilde{v}}\| \cos{\theta_{\bm{x},\bm{\tilde{v}} } }},
\end{align*}
where $\theta_{\bm{x},\bm{\tilde{v}} }$ denotes the angle between vectors $\bm{x}$ and $\bm{\tilde{v}}$. Since we require $w_1\geq 0$, this implies that in turn we are looking at the case, \begin{align*} \alpha \geq \frac{ \| \bm{\tilde{g}} - \alpha \bm{\tilde{v}}\| } { \|\bm{\tilde{v}}\| \cos{\theta_{\bm{x},\bm{\tilde{v}} } }}. \end{align*} Since $\bm{x}$ can be arbitrarily chosen, select $\bm{x}$ so as to run parallel to $\bm{\tilde{v}}$, thereby making $\cos{\theta_{\bm{x},\bm{\tilde{v}} } }=1$. Thus, $\bm{x}$ becomes completely removed from this expression, completing the proof.
\end{proof}

Therefore we have access to a lower bound as a function of $\alpha$, using \textit{only} the left and right eigenvectors. However notice that unfortunately this definition is implicit, in that $\alpha$ appears on both the left, and right-hand sides. In order to establish an expression with $\alpha$ only on one side, some assumptions will need to be made. These will be clarified in Remark \ref{rem:extend}.

\begin{remark}\textbf{[Simplified $\bm{\alpha}$ Range]}\label{rem:extend}
If $w_1 \geq 0$, and $|w_1 - \alpha| < |w_2 - \alpha|$, then,  \begin{align}\alpha \geq \frac{1}{2}\rho(\bm{M}).\end{align} 
\end{remark}
\begin{proof}
\begin{align}
    \frac{ \| \bm{\tilde{g}} - \alpha \bm{\tilde{v}}\| } { \|\bm{\tilde{v}}\| } &= \frac{ \| (\bm{M} - \alpha\bm{I}) \bm{\tilde{v}}\| } { \|\bm{\tilde{v}}\| }\nonumber\\
    &\leq \frac{ \| (\bm{M} - \alpha\bm{I})\| \| \bm{\tilde{v}}\| } { \|\bm{\tilde{v}}\| }\nonumber\\
    &=  \| (\bm{M} - \alpha\bm{I})\| \nonumber\\
    &= \rho(\bm{M} - \alpha\bm{I}) 
\end{align}
where $\bm{v} \neq 0$. From the proof thus far we have by implication: $\alpha - \frac{ \| \bm{\tilde{g}} - \alpha \bm{\tilde{v}}\| } { \|\bm{\tilde{v}}\| } \geq \alpha -  \rho(\bm{M} - \alpha\bm{I})$. From Corollary \ref{cor:mine} it is known that $\alpha - \frac{ \| \bm{\tilde{g}} - \alpha \bm{\tilde{v}}\| } { \|\bm{\tilde{v}}\| }\geq 0$, if $w_1 \geq 0$, and so by transitivity it is natural to require that $\alpha -  \rho(\bm{M} - \alpha\bm{I})\geq 0$. The given condition of $|w_1 - \alpha| < |w_2 - \alpha|$ means that we can expand as follows: $\alpha -  \rho(\bm{M} - \alpha\bm{I})=2\alpha -\rho(\bm{M})$, which in turn implies that $\alpha \geq \frac{1}{2}\rho(\bm{M})$, thus completing the proof. 

\end{proof}

Ultimately, Remark \ref{rem:extend} tries to make clear that a reasonable value to choose for $\alpha$ in order to best estimate the value of an (unobserved) $w_1$ is that of: $\alpha = \frac{1}{2}\rho(\bm{M})$. This is because from Corollary \ref{cor:mine} we know that $\alpha - \frac{ \| \bm{\tilde{g}} - \alpha \bm{\tilde{v}}\| } { \|\bm{\tilde{v}}\| }\geq 0$ in order to ensure that $w_1\geq 0$, and the value $\alpha=\frac{1}{2}\rho(\bm{M})$ appears to also satisfy this condition (Remark \ref{rem:extend}). However, once again it seems that we have run into a problem! This $\alpha$ value still requires access to $\bm{M}$! In particular, to the access of $\rho(\bm{M})$. Luckily however, $\rho(\bm{M})$ (the largest eigenvalue in this case) is much easier to estimate than least eigenvalue, $w_1$, because matrix reconstruction procedures tend to be much more strongly driven by the presence of the \textit{larger} eigenvalues (i.e. the larger singular values). For example, since we already have the relationship: $\bm{g}_i= \bm{Mv}_i$ (Lemma \ref{lem:LR}), one could pose an estimate for $\bm{M}$ as: $\bm{M}\approx \bm{g}_i\bm{v}_i^{\dagger}$, where the dagger symbol denotes the pseudo-inverse. This estimation of $\bm{M}$ is much more strongly driven by the larger singular values, and in fact, it can be shown that the largest singular value of this pseudo-inverse estimate for $\bm{M}$, is upper-bounded by the largest singular value for the full $\bm{M}$. This idea is clarified in Lemma \ref{lem:sigma1}. In practice this means as one collects more left and right eigenvectors, one can obtain increasingly better estimates of $\alpha$.

\begin{lem}\label{lem:sigma1}
If $\bm{g} = \bm{Mv}$, then $\sigma_1(\bm{g} \bm{v}^{\dagger})\leq \sigma_1(\bm{M})\coloneqq\rho(\bm{M})$, where $\sigma_1(\bm{M})$ denotes the largest singular value of $\bm{M}$.
\end{lem}

\begin{proof}
Consider that $\bm{g} \bm{v}^{\dagger} =  \bm{g} (\bm{v}^{\intercal}\bm{v})^{-1}\bm{v}^{\intercal}$ by definition of the pseudo-inverse. Thus, 
\begin{align}
    \|\bm{g} (\bm{v}^{\intercal}\bm{v})^{-1}\bm{v}^{\intercal}\| &\leq  \|\bm{g}\|\cdot  \|(\bm{v}^{\intercal}\bm{v})^{-1}\| \cdot \|\bm{v}^{\intercal}\| \nonumber \\
   &= \|\bm{Mv}\|\cdot  \|\|\bm{v}\|^{-2}\| \cdot \|\bm{v}^{\intercal}\|\nonumber \\
   &\leq  \|\bm{M}\| \cdot \|\bm{v}\| \cdot  \|\bm{v}\|^{-2} \cdot \|\bm{v}^{\intercal}\| \nonumber\\
   &= \|\bm{M}\|\cdot\|\bm{v}\|^2 \cdot  \|(\bm{v})\|^{-2}  \nonumber \\
   &= \|\bm{M}\|\nonumber
\end{align}

Thus, $\|\bm{g} \bm{v}^{\dagger}\| \leq \|\bm{M}\|$. Moreover the induced operator norm of a matrix evaluates to its largest singular value (the spectral radius in the case of symmetric positive definite matrices) implying that, $\sigma_1(\bm{g} \bm{v}^{\dagger})\leq \sigma_1(\bm{M})=\rho(\bm{M})$ which completes the proof. 
\end{proof}

\textbf{Summary: } \textit{We have established conditions under which the bound of Theorem \ref{thm:mine} can be practically used. These conditions are developed by observing how $\alpha$ behaves around the required region of $w_1 \geq 0$. It seems that a good reference for $\bm{\alpha}$ is to be around $\frac{1}{2}\rho(\bm{M})$, and in turn, $\rho(\bm{M})$ can be approximated from knowledge that $\bm{M}\approx \bm{gv}^{\dagger}$.}

\section{Examples}

The aim of this section will be to demonstrate the usage of the developed lower-bounds. Ultimately, it will be shown that a reasonable value for $w_1$ can be inferred, without ever knowing the full mass matrix. In particular, consider the following mass and stiffness matrices,
\begin{align*}
     \bm{M}_1 = \left(\begin{matrix}15&0&0&0&0\\0&21&0&0&0\\0&0&24&0&0\\0&0&0&27&0\\0&0&0&0&30\end{matrix}\right) \quad  \bm{M}_2 = \left(\begin{matrix}30&0&0&0&0\\0&170&0&0&0\\0&0&180&0&0\\0&0&0&190&0\\0&0&0&0&200\end{matrix}\right)
\end{align*}
\begin{align*}
     \bm{K} = \left(\begin{matrix}k_1&-k_1&0&0&0\\-k_1&k_1+k_2&-k_2&0&0\\0&-k_2&k_2+k_3&-k_3&0\\0&0&-k_3&k_3+k_4&-k_4\\0&0&0&-k_4&k_5\end{matrix}\right)  
\end{align*}

where $k_i = 1000i$ for $i=\{1,..,5\}$. For the ensuing analysis, assume that we do not know these matrices, and that we only have access to the total $(\bm{M},\bm{K})$ system via the left and right eigenvectors only. We will name the mass systems which are being analyzed as: $\bm{M}_1$, and $\bm{M}_2$ respectively. The result of analysing systems $\bm{M}_1$ and $\bm{M}_2$ are shown in Figures \ref{fig:sys1}, and \ref{fig:sys2} respectively.

From Figure \ref{fig:sys1} it can be seen that the suggested procedure of (i) Using the approximation: $\bm{M'}\approx \bm{g_1v_1}^{\dagger}$\footnote{Note that in Figures \ref{fig:sys1}, and \ref{fig:sys2} the dagger symbol is replaced by ``T''.}, and then (ii) Estimating the corresponding $\alpha= \frac{1}{2}\rho(\bm{M'})$ value seems to coincide very closely to the $w_1=0$ point. This is a nice conclusion, but unfortunately it is not a very useful one, as trivially we already expect that $w_1\geq 0$ for the mass matrices. However, if one uses a \textit{more informed} version of $\alpha= \frac{1}{2}\rho(\bm{M'})$ by constructing the $\bm{M'}$ estimate by means of collecting more left and right eigenvectors, then the lower-bound on $w_1$ can be shown to increase. Interestingly, it would appear that in the case of system $\bm{M}_2$ if one has access to \textit{all} the eigenvectors then the approximation precisely coincides to the point at which the theory of the bound begins to collapse. This is the point where $|w_1 -\alpha | = |w_2 - \alpha|$, or in other words, where the value chosen for $\alpha$ is equidistant from $w_2$ and $w_1$. However in system $\bm{M}_1$, even with information of all the left and right eigenvectors, the lower-bound seems to undershoot this point. On this point, if one truly had all information on the left and right eigenvectors then practically speaking, there is no reason to even use this bound - as in principle the mass matrix could be completely reconstructed. However, this analysis was performed here to show the behavior of the lower-bound in such an extreme regime. Ultimately, this lower-bound was analyzed in three distinct regimes. That is, in the presence of: (i) The least information possible (only $v_1$, and $g_1$ are known), (ii) A moderate amount of information ($v_{1:3}$, and $g_{1:3}$), as well as (iii) Complete information ($v_{1:5}$, $g_{1:5}$). Indeed, this bound was derived with the purpose of being used in the presence of partial information (that is, in points (i) and (ii)), and from the toy problems it would appear that the case of having access to only $v_{1:3}$, and $g_{1:3}$ provides reasonable estimates for the lower-bound on the least eigenvalue, $w_1$. The value for the lower-bounds turned out to be 6.8 for system $\bm{M}_1$, and 18.22 for system $\bm{M}_2$ in this case. These systems had true unobserved eigenvalues of 15, and 30 respectively.

\begin{figure}[H]
    \centering
    \begin{subfigure}{.65\textwidth}
    \includegraphics[width=\linewidth]{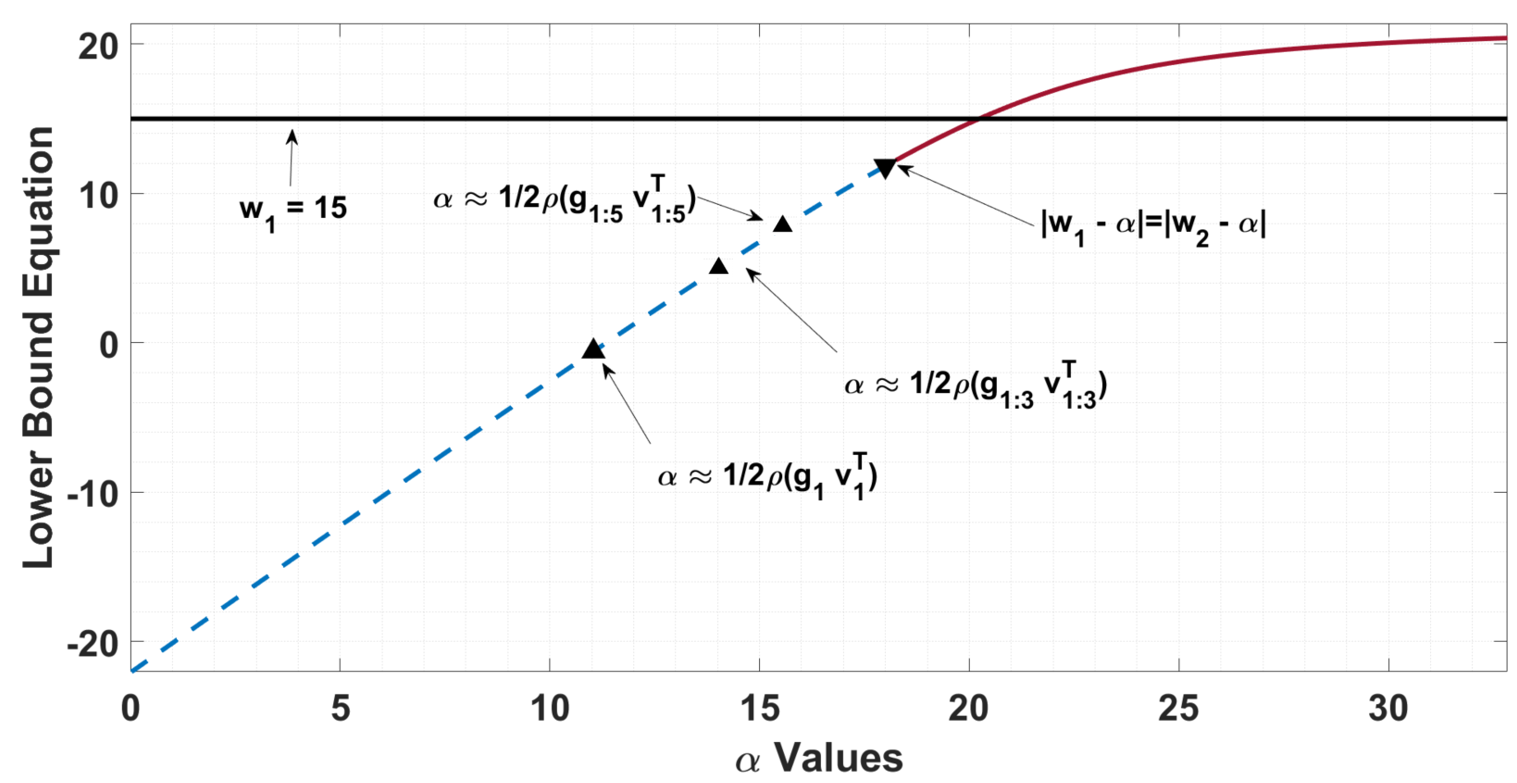}
    \caption{Analysis for system $\bm{M}_1$.}
    \label{fig:sys1}
    \end{subfigure}
    
    \begin{subfigure}{.65\textwidth}
    \includegraphics[width=\linewidth]{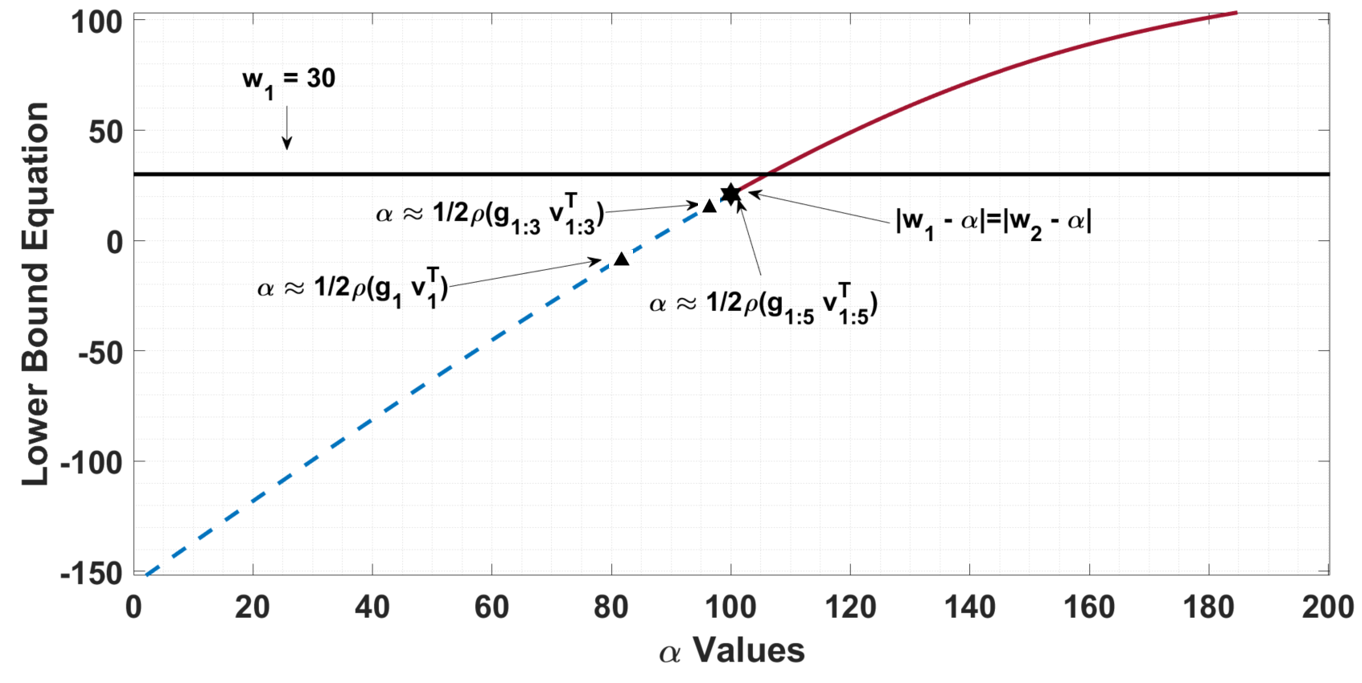}
    \caption{Analysis for system $\bm{M}_2$.}
    \label{fig:sys2}
    \end{subfigure}
    \caption{Varying $\alpha$ to see which $F(\alpha)$ value is able to best provide a lower-bound estimate for $w_1$ in each system, where $F(\alpha)=\alpha - \frac{\|\bm{\tilde{g}}-\alpha \bm{\tilde{v}}\|}{\|\bm{\tilde{v}}\|}$. Note that as more information is obtained (that is, more left and right eigenvectors are obtained), the lower-bound becomes increasingly better. In these figures, the blue dash represents the part of $F(\alpha)$ which has the condition $|w_1 -\alpha| < |w_2-\alpha|$ satisfied.}
\end{figure}

\section{Conclusion}
A general lower-bound (which varies as a function of a scalar, $\alpha$), was developed for the least eigenvalue of the mass matrix in linear-elastic systems that have undergone a negative perturbation. This lower-bound was developed in the absence of knowledge about the mass and stiffness matrices, and only requires to the first left, and right eigenvectors of the system - which are quantities that can be established from physical experiments. The range of $\alpha$ values which result in positive mass eigenvalues upon perturbing the system was also developed. An example of how to use these bounds was shown on two mass-spring systems. A reasonable value of $\alpha$ arose when considering the spectral radius of an estimated mass matrix. However, further research should be conducted in studying different ways to obtain optimal $\alpha$ values for even tighter lower bounds. Moreover, the toy problems studied here clearly indicate that a curvature change is present in the region of $|w_2-\alpha| < |w_1-\alpha|$. This implies that further investigations could also be performed to better anticipate the onset of the second least eigenvalue, $w_2$ (with respect to this ensuing curvature), thereby leading to a more complete understanding of partially observed mass-spring systems.

\bibliography{References}
\end{document}